\documentclass[notitlepage,onecolumn,amsmath,nobibnotes,nofootinbib,superscriptaddress,preprintnumbers]{revtex4-1}
\usepackage{hyperref}
\hypersetup{breaklinks,colorlinks,citecolor=blue}

\usepackage{dsfont}
\usepackage{amssymb}
\usepackage{amsmath}
\usepackage{mathrsfs}
\usepackage{amsthm}
\usepackage{color}
\usepackage{xcolor}
\usepackage{bold-extra}
\usepackage{eufrak}

\numberwithin{equation}{section}
\newtheorem{definition}{Definition}[section]
\newtheorem{proposition}{Proposition}[section]

\newtheorem{theorem}{Theorem}[section]
\newtheorem{remark}{Remark}[section]
\newtheorem*{openq*}{\color{red}{Open Question}}

\newcommand{\E}{\mathbb{E}}

\newcommand{\cont}[1]{\stackrel{#1}{\frown}}
\newcommand{\smcont}[1]{\star_{#1}^{#1 -1}}

\newcommand{\R}{\mathbb{R}}

\newcommand{\norm}[1]{\left\lVert#1\right\rVert}
\newcommand{\abs}[1]{\left\lvert#1\right\rvert}
\newcommand{\curly}[1]{\left\lbrace#1\right\rbrace}
\newcommand{\innerprod}[1]{\left\langle#1\right\rangle}

\begin{document}

\title{\Large \bfseries{\scshape{Vector-valued semicircular limits on the free Poisson chaos}}}

\author{Solesne Bourguin}
\email{solesne.bourguin@gmail.com}
\affiliation{Department of Mathematics and Statistics, Boston University, 111 Cummington Mall, Boston, MA 02215, USA}

 \begin{abstract}
~\\
\noindent \textit{Abstract:} In this note, we prove a multidimensional counterpart of the central limit theorem on the free Poisson chaos recently proved by Bourguin and Peccati (2014). A noteworthy property of convergence toward the semicircular distribution on the free Poisson chaos is obtained as part of the limit theorem: component-wise convergence of sequences of vectors of multiple integrals with respect to a free Poisson random measure toward the semicircular distribution implies joint convergence. This result complements similar findings for the Wiener chaos by Peccati and Tudor (2005), the classical Poisson chaos by Peccati and Zheng (2010) and the Wigner chaos by Nourdin, Peccati and Speicher (2013).
~
 \end{abstract}


\maketitle

\section{Introduction and background}
\noindent Let $\left\lbrace W_t\colon t \geq 0 \right\rbrace $ be a standard Brownian motion on $\R_{+}$ and let $n \geq 1$ be an integer. Denote by $I_{n}^{W}\left( f\right)$ the multiple stochastic Wiener-It\^o integral of order $n$ of a symmetric function $f \in L^2\left( \R_{+}^{n}\right)$. Denote by $L^2_s\left( \R_{+}^{n}\right)$ the subset of $L^2\left( \R_{+}^{n}\right)$ composed of symmetric functions. The collection of random variables $\left\lbrace I_{n}^{W}\left( f\right) \colon f \in L^2_s\left( \R_{+}^{n}\right) \right\rbrace $ is what is usually called the $n$-th Wiener chaos associated with $W$. 
\\~\\
In a seminal paper of 2005, Nualart and Peccati \cite{nuapec1} proved that convergence to the standard normal distribution of an sequence of elements with variance one living inside a fixed Wiener chaos was equivalent to the convergence of the fourth moment of this sequence to three. This result, now known as the \textit{fourth moment theorem}, was proved to hold as well for sequences of vectors of multiple integrals possibly of different orders by Peccati and Tudor \cite{petu1} shortly after. On top of the multidimensional central limit theorem, they discovered that for sequences of fixed order chaos elements, component-wise convergence to the Gaussian distribution always implies joint convergence.  More specifically, these properties can be stated in the following way.
\begin{theorem}[Peccati and Tudor \cite{petu1}, 2005]
\label{wiener4thMomentTheorem}
Let $d \geq 2$ and $n_1, \ldots , n_{d}$ be integers, and let $\left\lbrace \left( F_{k}^{(1)}, \ldots , F_{k}^{(d)}\right) \colon k \geq 1\right\rbrace $ be a sequence of random vectors such that, for every $i = 1, \ldots , d$, the random variable $F_{k}^{(i)}$ lives inside the $n_i$-th Wiener chaos associated with $W$. Assume that for every $i,j = 1, \ldots , d$, $\E\left[F_{k}^{(i)}F_{k}^{(j)} \right] \underset{k \rightarrow \infty}{\longrightarrow} c(i,j)$, where $c = \left\lbrace c(i,j) \colon i,j = 1, \ldots , d \right\rbrace $ is a positive definite symmetric matrix. Then, the following three assertions are equivalent, as $k \rightarrow \infty$.
\begin{enumerate}
\item[(i)] The sequence of vectors $\left( F_{k}^{(1)}, \ldots , F_{k}^{(d)}\right)$ converges to the $d$-dimensional Gaussian distribution $\mathcal{N}(0,c)$.
\item[(ii)] For every $i = 1, \ldots , d$,  the sequence of random variables $F_{k}^{(i)}$ converges to the Gaussian distribution $\mathcal{N}(0,c(i,i))$.
\item[(iii)] For every $i = 1, \ldots , d$, $\E\left[ \left( F_{k}^{(i)}\right) ^4\right] \longrightarrow 3c(i,i)^2 = \E\left[ \mathcal{N}(0,c(i,i)) ^4\right]$.
\end{enumerate}
\end{theorem}
\noindent Non-commutative counterparts of these theorems have been established in the context of the chaos associated with a free Brownian motion $\left\lbrace S_t\colon t \geq 0 \right\rbrace $ in \cite{kenopesp1} for the one-dimensional limit theorem and in \cite{nopesp1} for the multidimensional version for which the authors proved that the property that component-wise convergence implies joint convergence holds as in the classical Brownian case. This result can be stated in the following way.
\begin{theorem}[Nourdin, Peccati and Speicher \cite{nopesp1}, 2013]
\label{wigner4thMomentTheorem}
Let $d \geq 2$ and $n_1, \ldots , n_{d}$ be integers, and consider a positive definite symmetric matrix  $c = \left\lbrace c(i,j) \colon i,j = 1, \ldots , d \right\rbrace $. Let $\left(s_1, \ldots , s_d \right) $ be a semicircular family with covariance $c$ (see Definition \ref{semicircfamily}). For each $i = 1, \ldots , d$, we consider a sequence $\left\lbrace f_{k}^{(i)} \colon k\geq 1\right\rbrace $ of mirror-symmetric (see Definition \ref{mirrorsymdef}) functions in $L^2\left( \R_{+}^{n_{i}}\right) $ such that, for all $i,j = 1, \dots , d$, $$\varphi\left[  I_{n_{i}}^{S}\left( f_{k}^{(i)}\right) I_{n_{j}}^{S}\left( f_{k}^{(j)}\right)\right]  \underset{k \rightarrow \infty}{\longrightarrow} c(i,j).$$ Then, the following three assertions are equivalent, as $k \rightarrow \infty$.
\begin{enumerate}
\item[(i)] The vector $\left( I_{n_{1}}^{S}\left( f_{k}^{(1)}\right), \ldots , I_{n_{d}}^{S}\left( f_{k}^{(d)}\right)\right)$ converges in distribution to $\left(s_1, \ldots , s_d \right)$.
\item[(ii)] For every $i = 1, \ldots , d$,  the random variable $I_{n_{i}}^{S}\left( f_{k}^{(i)}\right)$ converges to $s_i$.
\item[(iii)] For every $i = 1, \ldots , d$, $\varphi\left[ I_{n_{i}}^{S}\left( f_{k}^{(i)}\right) ^4\right] \longrightarrow 2c(i,i)^2 = \varphi\left[ s_i ^4\right]$.
\end{enumerate}
\end{theorem}
\noindent These results, both in the classical as well as in the free case, provide a dramatic simplification to the more conventional method of moments and cumulants in the sense that it is enough to control the second and the fourth moments of the components of a sequence of vector-valued multiple integrals with respect to a classical or free Brownian motion to ensure convergence to a central limit. These theorems have led to a wide collection of new results and inspired several new research directions -- see the following constantly updated webpage that lists all the results directly connected to the fourth moment theorems or the techniques devellopped in their proofs. 
\begin{center}
\href{https://sites.google.com/site/malliavinstein/home}{https://sites.google.com/site/malliavinstein/home}
\end{center}
In the classical probability setting, the quest for similar results and properties in the framework of the chaoses associated to a Poisson random measure have lead to a wide range of results, both theoretical as well as applied to fields as diverse as stochastic geometry (see e.g. \cite{bope2,laspenschtha1,resc1,sch1,schtha2,schtha1}) or cosmological statistics (see e.g. \cite{bodumape1}). Recently, Bourguin and Peccati proved in \cite{bope1} that a fourth moment theorem holds on the chaos associated with a free Poisson random measure. In view of this result, a natural question is to assess whether or not a multidimensional version of this fourth moment theorem can be established and if the property that component-wise convergence to the semicircular distribution always implies joint convergence holds as is the case on the Brownian and free Brownian chaoses.
\\~\\
This note provides a positive answer to this question (see Theorem \ref{maintheorem}), hence completing the line of chaotic multidimensional limit theorems initiated in \cite{petu1} and further developed in \cite{pezh1, nopesp1}.
\begin{theorem}
\label{maintheorem}
Let $d\geq 2$ and $n_1,\ldots , n_d$ be some fixed integers. Let $c = \curly{c(i,j) \colon i,j=1,\ldots ,d}$ be a positive definite symmetric matrix and consider $\left(s_1,\ldots, s_d\right)$ to be a semicircular family with covariance $c$ (see Definition \ref{semicircfamily}). For each $i=1,\dots,d$, let $\curly{f_{k}^{(i)} \colon k \geq 1}$ be a sequence of tamed (see Definition \ref{tameddef}) mirror-symmetric (see Definition \ref{mirrorsymdef}) functions in $L^{2}\left(\R_{+}^{n_i}\right)$ such that, for all $i,j =1,\ldots ,d$, 
\begin{equation*}
	\varphi\left[I_{n_i}^{\hat{N}}\left(f_{k}^{(i)}\right)I_{n_j}^{\hat{N}}\left(f_{k}^{(j)}\right)\right] \underset{k \rightarrow \infty}{\longrightarrow} c(i,j).
\end{equation*}
Then, the following three statements are equivalent.
\begin{enumerate}
\item[(i)] The vector $\left(I_{n_1}^{\hat{N}}\left(f_{k}^{(1)}\right), \ldots ,I_{n_d}^{\hat{N}}\left(f_{k}^{(d)}\right) \right)$  converges in distribution to $\left(s_1,\ldots, s_d\right)$.
\item[(ii)] For each $i=1,\dots,d$, the random variable $I_{n_i}^{\hat{N}}\left(f_{k}^{(i)}\right)$ converges in distribution to $s_i$.
\item[(iii)] For each $i=1,\dots,d$, 
\begin{equation*}
	\varphi\left[I_{n_i}^{\hat{N}}\left(f_{k}^{(i)}\right)^4\right] \underset{k \rightarrow \infty}{\longrightarrow} 2c(i,i)^2.
\end{equation*}
\end{enumerate}
\end{theorem}
\begin{remark}
The notion of a {\it tamed sequence} of functions appearing in Theorem \ref{maintheorem} and formally defined in Definition \ref{tameddef} is needed in order to deal with the complicated combinatorial structures arising from the computation of moments. This notion was first introduced in \cite{bope1} and is similar to the uniform boundedness assumption on the sequences of functions used in \cite[Proposition 3.1]{nopesp1}.
\end{remark}
\noindent Combining Theorem \ref{maintheorem} with \cite[Theorem 1.2, Point (B)]{bope1}, one obtains the following counterexample to the multidimensional transfer principle proved to hold in the Brownian and free Brownian framework in \cite[Theorem 1.6]{nopesp1}.
\begin{theorem}
Let $d \geq 2$ and let $n_1, \ldots , n_d$ be some fixed integers such that $n_1 + \cdots + n_d > d$. Let $\hat{\eta}$ denote a centered Poisson process with Lebesgue control measure. Then, there exist sequences $\left\lbrace g_{k}^{(i)} \colon k \geq 1\right\rbrace $, $i = 1 , \ldots ,d$, such that, as $n \rightarrow \infty$, the vector $\left(I_{n_1}^{\hat{N}}\left(g_{k}^{(1)}\right), \ldots ,I_{n_d}^{\hat{N}}\left(g_{k}^{(d)}\right) \right)$  converges in distribution to a $d$-dimensional semicircular distribution while the vector $\left(I_{n_1}^{\hat{\eta}}\left(g_{k}^{(1)}\right), \ldots ,I_{n_d}^{\hat{\eta}}\left(g_{k}^{(d)}\right) \right)$  converges in distribution to a $d$-dimensional Poisson distribution.
\end{theorem}
\noindent The remainder of the paper is organized as follows. Section \ref{sectionrelelements} contains the most relevant elements of free probability theory needed in order to make this note as self-contained as possible. Theorem \ref{maintheorem} is proved in Section \ref{proofmaintheo} while Section \ref{auxsection} gathers technical results used in the proof of Theorem \ref{maintheorem}.
\section{Relevant elements of free probability}
\label{sectionrelelements}

\noindent This section list the most relevant elements of free probability theory used in the present note. For more details on the tools and notions used below, the reader is referred to the references \cite{bope1} and \cite{nisp}.
\subsection{Free probability, free Poisson process and stochastic integrals}
\noindent Let $\left( \mathscr{A}, \varphi \right) $ be a tracial $W^*$-probability space, that is $\mathscr{A}$ is a von Neumann algebra with involution $*$ and $\varphi \colon \mathscr{A} \rightarrow \mathbb{C}$ is a unital linear functional assumed to be weakly continuous, positive (meaning that $\varphi\left( X\right) \geq 0$ whenever  $X$ is a non-negative element of $\mathscr{A}$), faithful (meaning that $\varphi\left(XX^* \right) = 0 \Rightarrow X = 0$ for every $X \in \mathscr{A}$) and tracial (meaning that $\varphi\left(XY \right) = \varphi\left(YX \right) $ for all $X,Y \in \mathscr{A}$). The self-adjoint elements of $\mathscr{A}$ will be referred as random variables. Given a random variable $X \in \mathscr{A}$, the law of $X$ is defined, as in \cite[Proposition 3.13]{nisp}, to be the unique Borel measure on $\R$ having the same moments as $X$. The non-commutative space $L^2(\mathscr{A},\varphi)$ denotes the completion of $\mathscr{A}$ with respect to the norm $\norm{X}_2 = \sqrt{\varphi\left( XX^* \right) }$.
\begin{definition}
A collection of random variables $X_1, \ldots , X_n$ on $\left( \mathscr{A}, \varphi \right) $ is said to be free if $$\varphi\left( \left[P_1\left(X_{i_1}\right) - \varphi\left( P_1\left(X_{i_1}\right)\right)   \right] \cdots \left[P_m\left(X_{i_m}\right) - \varphi\left( P_m\left(X_{i_m}\right)\right)   \right] \right) = 0 $$ whenever $P_1, \ldots , P_m$ are polynomials and $i_1, \ldots , i_m \in \left\lbrace 1, \ldots, n\right\rbrace $ are indices with no two adjacent $i_j$ equal.
\end{definition}
\begin{definition}
The centered semicircular distribution with variance $t>0$, denoted by $\mathcal{S}(0,t)$, is the probability distribution given by $$\mathcal{S}(0,t)(dx) = (2\pi t)^{-1}\sqrt{4t-x^2}dx, \quad |x|< 2\sqrt{t}.$$ The symmetric aspect of this distribution around zero guaranties that all its odd moments are zero. Furthermore, it is straightforward to check that the even moments are given, for any non-negative integer $m$, by $\varphi\left(\mathcal{S}(0,t)^{2m} \right) = C_m t^m$, where $C_m = \frac{1}{m+1}\binom{2m}{m}$ is the $m$-th Catalan number. 
\end{definition}
\begin{definition}
The free Poisson distribution with rate $\lambda >0 $, denoted by $P(\lambda)$, is the probability distribution defined 
as follows: (i) if $\lambda\in (0,1]$, then $P(\lambda) = (1-\lambda)\delta_0 + \lambda\widetilde{\nu}$, and (ii) if $\lambda > 1$, then 
$P(\lambda) = \widetilde{\nu}$, where $\delta_0$ stands for the Dirac mass at $0$. Here, $\widetilde{\nu}(dx) = 
(2\pi x)^{-1}\sqrt{4\lambda-(x-1-\lambda)^2}dx, \, x\in \big((1-\sqrt{\lambda})^2,(1+\sqrt{\lambda})^2\big)$.
\end{definition}
\begin{definition}
A free Poisson process $N$ consists of: (i) a filtration $\left\lbrace \mathscr{A}_t \colon t \geq 0 \right\rbrace $ of von Neumann sub-algebras of $\mathscr{A}$ (in particular, $\mathscr{A}_s \subset \mathscr{A}_t$ for $0 \leq s < t$), (ii) a collection $N = \left\lbrace N_t \colon t\geq 0\right\rbrace $ of self-adjoint operators in $\mathscr{A}_{+}$ ($\mathscr{A}_{+}$ denotes the cone of positive operators in $\mathscr{A}$) such that: (a) $N_0 = 0$ and $N_t \in \mathscr{A}_t$ for all $t \geq 0$, (b) for all $t \geq 0$, $N_t$ has a free Poisson distribution with rate $t$, and (c) for all $0 \leq u < t$, the increment $N_t - N_u$ is free with respect to $\mathscr{A}_u$, and has a free Poisson distribution with rate $t-u$. $\hat N$ will denote the collection of random variables $\hat{N}= \left\lbrace  \hat{N}_t = N_t - t\mathbf{1} \colon t\geq 0 \right\rbrace $, where $\mathbf{1}$ stands for the unit of $\mathscr{A}$. $\hat N$ will be referred to as a compensated free Poisson process.
\end{definition}
\noindent For every integer $n\geq 1$, the space $L^2\left( \R_{+}^n;\mathbb{C}\right) = L^2\left( \R_{+}^n\right)$ denotes the collection of all complex-valued functions on $\R_{+}^n$ that are square-integrable with respect to the Lebesgue measure on $\R_{+}^n$. 
\begin{definition}
\label{mirrorsymdef}
Let $n$ be a natural number and let $f$ be a function in $L^2\left( \R_{+}^n\right)$. 
\begin{enumerate}
\item The adjoint of $f$ is the function $f^{\ast}\left(t_1, \ldots , t_n \right) = \overline{f\left(t_n, \ldots , t_1\right)}$.
\item The function $f$ is called mirror-symmetric if $f = f^{\ast}$, i.e., if $f\left(t_1, \ldots , t_n \right) = \overline{f\left(t_n, \ldots , t_1\right)}$ for almost all $\left( t_1,\ldots , t_n\right) \in \R_{+}^{n}$ with respect to the product Lebesgue measure.
\item The function $f$ is called fully symmetric if it is real-valued and, for any permutation $\sigma$ in the symmetric group $\mathfrak{S}_{n}$, $f\left( t_1, \ldots , t_n\right) = f\left( t_{\sigma(1)}, \ldots , t_{\sigma(n)}\right) $ for almost all $\left( t_1,\ldots , t_n\right) \in \R_{+}^{n}$ with respect to the product Lebesgue measure.
\end{enumerate}
\end{definition}
\begin{definition}
\label{defcontractions}
Let $n,m$ be natural numbers and let $f \in L^2\left( \R_{+}^n\right)$ and $g \in L^2\left( \R_{+}^m\right)$. Let $p \leq n \wedge m$ be a natural number. The $p$-th arc contraction $f \cont{p} g$ of $f$ and $g$ is the $L^2\left( \R_{+}^{n+m-2p}\right)$ function defined by nested integration of the middle $p$ variables in $f \otimes g$:
\begin{equation*}
f  \cont{p} g (t_1,\ldots, t_{n+m - 2p}) =\int_{\R_{+}^{p}}f(t_1,\ldots , t_{n-p},s_1, \ldots , s_p)g(s_p, \ldots , s_1 , t_{n-p+1},\ldots, t_{n+m-2p})ds_1 \cdots ds_p.
\end{equation*}
In the case where $p=0$, the function $f  \cont{0} g$ is just given by $f \otimes g$. Similarly, the $p$-th star contraction $f \smcont{p} g$ of $f$ and $g$ is the $L^2\left( \R_{+}^{n+m-2p+1}\right)$ function defined by nested integration of the middle $p-1$ variables and identification of the first non-integrated variable in $f \otimes g$:
\begin{equation*}
f  \smcont{p} g (t_1,\ldots, t_{n+m - 2p+1}) =\int_{\R_{+}^{p-1}}f(t_1,\ldots , t_{n-p+1},s_1, \ldots , s_{p-1})g(s_{p-1}, \ldots , s_1 , t_{n-p+1},\ldots, t_{n+m-2p+1})ds_1 \cdots ds_{p-1}.
\end{equation*}
\end{definition}
\noindent For $f \in L^2\left(\R_{+}^n\right)$, we denote by $I_{n}^{\hat{N}}(f)$ the multiple integral of $f$ with respect to $\hat{N}$. The space $L^2(\mathcal{X}({\hat{N}}), \varphi) = \{I^{\hat{N}}_n(f) : f\in L^2(\R_{+}^n), n\geq 0\}$ is a unital $\ast$-algebra, with product rule given, for any $n,m\geq 1$, $f \in L^2\left(\R_{+}^n\right)$, $g \in L^2\left(\R_{+}^m\right)$, by
\begin{equation*}
I^{\hat{N}}_n(f)I^{\hat{N}}_m(g) = \sum_{p=0}^{n \wedge m} I^{\hat{N}}_{n+m-2p}\left( f \cont{p} g\right) +  \sum_{p=1}^{n \wedge m} I^{\hat{N}}_{m+n-2p+1}\left( f \smcont{p} g\right)
\end{equation*}
and involution $I_n^{\hat{N}}(f)^{\ast} = I_n^{\hat{N}}(f^{\ast})$. 
\begin{remark}
Observe that it follows from the definition of the involution on the algebra $L^2(\mathcal{X}(\hat{N}), \varphi)$ that operators of the type $I_n^{\hat{N}}(f)$ are self-adjoint if and only if $f$ is mirror-symmetric.
\end{remark}
\begin{definition}
\label{tameddef}
Let $n\geq 1$ be an integer. We say that a sequence $\{f_k : k\geq 1\}\subset L^2(\R_{+}^n)$ is tamed if the following conditions hold: every $f_k$ is bounded and has bounded support and, for every $p\geq 2$ and every $\pi \in \mathcal{P}\left( \bigotimes_{j=1}^{p}n\right) $, the numerical sequence 
\begin{equation*}
\left\lbrace \int_{\pi} \bigotimes_{j=1}^{p}\left| f_k\right| \colon k \geq 1\right\rbrace 
\end{equation*}
is bounded.
\end{definition}
\begin{remark}
There exists sufficient conditions in order for a sequence $\{f_k \colon k \geq 1\}$ to be tamed. It basically consists in requiring that $\{f_k \colon k \geq 1\}$ concentrates asymptotically, without exploding, around a hyperdiagonal: fix $n\geq 2$, and consider a sequence $\{f_k \colon k\geq 1\}\subset L^2(\R_{+}^n)$. Assume that there exist strictly positive numerical sequences $\{M_k, z_k, \alpha_k \colon k\geq 1\}$ such that $\alpha_k/z_k \to 0$ as $k \rightarrow \infty$ and the following properties are satisfied: (a) the support of $f_k$ is contained in the set $\prod_{j=1}^{n}\left( -z_k,z_k \right) $, (b) $|f_k|\leq M_k$, (c) $f_k(t_1,\ldots,t_n) = 0$, whenever there exist $t_i, t_j$ such that $\vert t_i - t_j\vert >\alpha_k$ and (d) for every integer $p\geq n$, the mapping $k\mapsto M_k^p z_k \alpha_k^{p-1}$ is bounded. Then, the sequence $\left\lbrace f_k \colon k\geq 1 \right\rbrace $ is tamed.
\end{remark}

\subsection{Non-crossing partitions, partition integrals and semicircular families}
\noindent A partition of $\left[n \right] = \left\lbrace 1, \ldots , n \right\rbrace $ is a collection of mutually disjoint subsets $b_1 , \ldots , b_r$ of $\left[ n\right] $ such that $B_1 \sqcup \cdots \sqcup B_r = \left[ n\right] $. The subsets are called the blocks of the partition and by convention, we order them by their least elements, i.e., $\min B_i < \min B_j $ if and only if $i < j$. The cardinality of a block $B$ is denoted by $\vert B \vert$. A block is said to be a singleton if it has cardinality one. A partition with only blocks of cardinality two is called a pairing. The set of all partitions of $\left[ n\right] $ is denoted $\mathscr{P}(n)$, the set of all pairings is denoted $\mathscr{P}_2(n)$, the set of all partitions without singletons is denoted $\mathscr{P}_{\geq 2}(n)$ and the set of all partitions without singletons and with at least one block of cardinality greater or equal to three is denoted $\mathscr{P}_{\geq 2+}(n)$. Observe that it holds that $ \mathscr{P}_2(n) \subset \mathscr{P}_{\geq 2}(n) \subset \mathscr{P}(n)$ and $\mathscr{P}_2(n) \sqcup \mathscr{P}_{\geq 2+}(n) =\mathscr{P}_{\geq 2}(n)$. The number of blocks of a partition $\pi \in \mathscr{P}(n)$ is denoted by $\vert \pi \vert$.
\begin{definition}
Let $\pi \in \mathscr{P}(n)$ be a partition of $\left[ n\right] $. $\pi$ is said to have a crossing if there are two distinct blocks $B_1, B_2$ in $\pi$ with elements $x_1,y_1 \in B_1$ and $x_2,y_2\in B_2$ such that $x_1 < x_2 < y_1 < y_2$. If $\pi \in \mathscr{P}(n)$ has no crossings, it is said to be a non-crossing partition. The set of non-crossing partitions of $\left[ n\right] $ is denoted $NC(n)$. The non-crossing elements of $\mathscr{P}_2(n)$, $\mathscr{P}_{\geq 2}(n)$ and $\mathscr{P}_{\geq 2+}(n)$ are denoted respectively by $NC_{\geq 2+}(n)$, $NC_{\geq 2+}(n)$ and $NC_{\geq 2+}(n)$. In that case too, it holds that $ NC_2(n) \subset NC_{\geq 2}(n) \subset NC_(n)$ and $NC_2(n) \sqcup NC_{\geq 2+}(n) = NC_{\geq 2}(n)$.
\end{definition}
\begin{definition}
\label{defpartition1}
Let $n_1, \ldots ,n_r$ be positive integers with $n = n_1 + \cdots + n_r$ and partition the set $\left[ n\right] $ according to these integers by putting $\left[ n\right] = B_1 \sqcup \cdots \sqcup B_r$, where $B_1 = \left\lbrace 1, \ldots , n_1 \right\rbrace $, $B_2 = \left\lbrace n_1 + 1, \ldots , n_1 +n_2 \right\rbrace $ and so forth through $B_r = \left\lbrace n_1 + \cdots + n_{r-1} + 1, \ldots , n_1 + \cdots + n_r \right\rbrace $. Denote this partition by $n_1 \otimes \cdots \otimes n_r$. A partition $\pi \in \mathscr{P}_{\geq 2}(n)$ is said to respect $n_1 \otimes \cdots \otimes n_r$ if no block of $\pi$ contains more than one element from any given block of $n_1 \otimes \cdots \otimes n_r$. For any given subset $\mathscr{Q}(n) \subset \mathscr{P}(n)$, the subset of $\mathscr{Q}(n)$ consisting of all partitions that respect $n_1 \otimes \cdots \otimes n_r$ is denoted $\mathscr{Q}\left( n_1 \otimes \cdots \otimes n_r\right) $.
\end{definition}
\begin{definition}
\label{deflink}
Let $n_1, \ldots ,n_r$ be positive integers and let $\pi \in \mathscr{P}_{\geq 2}\left(n_1 \otimes \cdots \otimes n_r \right)$. Let $B_1, B_2$ be two blocks in $n_1 \otimes \cdots \otimes n_r$. $\pi$ is said to link $b_1$ and $B_2$ if there is a block of $\pi$ containing an element of $B_1$ and an element of $B_2$. Define a graph $C_{\pi}$ whose vertices are the blocks of $n_1 \otimes \cdots \otimes n_r$; $C_{\pi}$ has an edge between $B_1$ and $B_2$ if and only if $\pi$ links $B_1$ and $B_2$. The partition $\pi$ on which the graph is based will be said to be connected with respect to $n_1 \otimes \cdots \otimes n_r$ (or that $\pi$ connects the blocks of $n_1 \otimes \cdots \otimes n_r$) if the graph $C_{\pi}$ is connected. The set $NC_{\geq 2}^{c}\left( n_1 \otimes \cdots \otimes n_r \right) $ will denote the subset of all the partitions in $NC_{\geq 2}^{c}\left( n_1 \otimes \cdots \otimes n_r \right) $ that both respect and connect $n_1 \otimes \cdots \otimes n_r$. Similarly, the set $NC_{\geq 2+}^{c}\left( n_1 \otimes \cdots \otimes n_r \right) $ will denote the subset of all the partitions in $NC_{\geq 2+}\left( n_1 \otimes \cdots \otimes n_r \right) $ that both respect and connect $n_1 \otimes \cdots \otimes n_r$.
\end{definition}
\begin{definition}
\label{defpartitionfunctions}
Let $n \geq 2$ be an integer and let $\pi \in \mathscr{P}_{\geq 2}(n)$. Let $f \colon \R_{+}^{n} \rightarrow \mathbb{C}$ be a measurable function. The partition integral of $f$ with respect to $\pi$, denoted $\int_{\pi}f$, is defined, when it exists, to be the constant
\begin{equation*}
\int_{\pi}f = \int_{\R_{+}^{n}}f\left( t_1,\ldots , t_n\right)\prod_{B \in \pi}\prod_{\left\lbrace i,j\right\rbrace \subset B}\delta\left( t_i - t_j\right)dt_1 \cdots dt_n.  
\end{equation*} 
\end{definition}
\begin{definition}
\label{semicircfamily}
Let $d \geq 2$ be an integer and let $c = \left\lbrace c(i,j) \colon i,j =1, \ldots , d \right\rbrace $ be a positive-definite symmetric matrix. A $d$-dimensional vector $\left(s_1 , \ldots , s_d \right) $ of random variables in $\mathscr{A}$ is said to be a semicircular family with covariance $c$ if for every $n \geq 1$ and every $\left(i_1, \ldots , i_n \right) \in \left[ d\right] ^n $,
\begin{equation*}
\varphi\left(s_{i_1}s_{i_2} \cdots s_{i_n} \right) = \sum_{\pi \in NC_{2}(n)}\prod_{\left\lbrace a,b \right\rbrace \in \pi } c\left( i_{a},i_{b}\right). 
\end{equation*} 
The previous relation implies in particular that, for every $i = 1, \ldots , d$, the random variable $s_i$ has a semicircular distribution with mean zero and variance $c(i,i)$.
\end{definition}

\section{Proof of Theorem \ref{maintheorem}}
\label{proofmaintheo}
\noindent Observe that the equivalence between \textit{(ii)} and \textit{(iii)} is a direct consequence of \cite[Theorem 4.3]{bope1}. As it is clear that \textit{(i)} implies \textit{(iii)}, we are left with proving that \textit{(iii)} implies \textit{(i)}. Assume that \textit{(iii)} holds and recall that by \cite[Theorem 4.3]{bope1}, this is equivalent to the fact that, for each $i = 1, \ldots ,d$, for all $\ell \in \left\lbrace 1,\ldots , n_i -1 \right\rbrace$ and for all $q \in \left\lbrace 0, \ldots, n_i-1\right\rbrace $, we have 
\begin{equation}
\label{assumpintheproof}
f_{k}^{(i)} \cont{\ell} f_{k}^{(i)} \underset{k \rightarrow \infty}{\longrightarrow} 0  \quad \mbox{in}\quad   L^{2}\left(\R_{+}^{2n_i - 2\ell}\right)                        \qquad \mbox{and}         \qquad                             f_{k}^{(i)}\star_{q+1}^{q}f_{k}^{(i)} \underset{k \rightarrow \infty}{\longrightarrow} 0 \quad \mbox{in}\quad   L^{2}\left(\R_{+}^{2n_i - 2q -1}\right).
\end{equation}
Observe that by a straightforward application of Fubini's Theorem along with the mirror-symmetry of the functions $f_{k}^{(i)}$, it holds that $\norm{f_{k}^{(i)}\star_{q}^{q-1}f_{k}^{(i)}}_{L^2\left( \R_{+}\right) } = \norm{f_{k}^{(i)}\star_{1}^{0}f_{k}^{(i)}}_{L^2\left( \R_{+}^{2n_i -1}\right) }$, so that we also have that, for each $i = 1, \ldots ,d$, 
\begin{equation*}
f_{k}^{(i_s)}\star_{1}^{0}f_{k}^{(i)} \underset{k \rightarrow \infty}{\longrightarrow} 0 \quad \mbox{in}\quad   L^{2}\left(\R_{+}^{2n_i -1}\right).
\end{equation*}
In order to show \textit{(i)}, we have to show that any moment in the variables $I_{n_1}^{\hat{N}}\left(f_{k}^{(1)}\right), \ldots , I_{n_d}^{\hat{N}}\left(f_{k}^{(d)}\right)$ converges, as $k$ goes to infinity, to the corresponding moment of the semicircular family $\left( s_1, \ldots , s_d\right) $. Let $r \geq 1$ and $i_1, \ldots , i_r$ be positive integers. Consider the moments 
\begin{equation*}
\varphi\left[I_{n_{i_1}}^{\hat{N}}\left(f_{k}^{(i_1) }\right) \cdots I_{n_{i_r}}^{\hat{N}}\left(f_{k}^{(i_r)}\right)\right].
\end{equation*}
The goal is to prove that these moments converge, as $k$ goes to infinity, to $\varphi\left[s_{i_1} \cdots s_{i_r}\right]$. By Proposition \ref{diagramformula}, we have 
\begin{equation*}
\varphi\left[I_{n_{i_1}}^{\hat{N}}\left(f_{k}^{(i_1) }\right) \cdots I_{n_{i_r}}^{\hat{N}}\left(f_{k}^{(i_r)}\right)\right] = \sum_{\pi \in NC_{\geq 2}\left(n_{i_1} \otimes \cdots \otimes n_{i_r}\right)}\int_{\pi}f_{k}^{(i_1)} \otimes \cdots \otimes f_{k}^{(i_r)}.
\end{equation*}
We can decompose the sum appearing on the right hand side by isolating the pairings and the rest of the partitions in $NC_{\geq 2}\left(n_{i_1} \otimes \cdots \otimes n_{i_r}\right)$, so that
\begin{equation*}
\varphi\left[I_{n_{i_1}}^{\hat{N}}\left(f_{k}^{(i_1) }\right) \cdots I_{n_{i_r}}^{\hat{N}}\left(f_{k}^{(i_r)}\right)\right] = \sum_{\pi \in NC_{2}\left(n_{i_1} \otimes \cdots \otimes n_{i_r}\right)}\int_{\pi}f_{k}^{(i_1)} \otimes \cdots \otimes f_{k}^{(i_r)} + \sum_{\pi \in NC_{\geq 2+}\left(n_{i_1} \otimes \cdots \otimes n_{i_r}\right)}\int_{\pi}f_{k}^{(i_1)} \otimes \cdots \otimes f_{k}^{(i_r)}.
\end{equation*}
In view of \eqref{assumpintheproof},  the argument used in \cite[Proof of Theorem 1.3]{nopesp1} guaranties that 
\begin{equation*}
\sum_{\pi \in NC_{2}\left(n_{i_1} \otimes \cdots \otimes n_{i_r}\right)}\int_{\pi}f_{k}^{(i_1)} \otimes \cdots \otimes f_{k}^{(i_r)} \underset{k \rightarrow \infty}{\longrightarrow} \sum_{\sigma \in NC_{2}(r)} \prod_{\left\lbrace s,t \right\rbrace \in \sigma } c\left( i_s, i_t\right),
\end{equation*}
which is exactly the moment $\varphi\left(s_{i_1}, \ldots, s_{i_r}\right) $ of a semicircular family with covariance matrix $c$. Therefore, in order to conclude the proof, it only remains to prove that 
\begin{equation}
\label{hastogotozero}
\sum_{\pi \in NC_{\geq 2+}\left(n_{i_1} \otimes \cdots \otimes n_{i_r}\right)}\int_{\pi}f_{k}^{(i_1)} \otimes \cdots \otimes f_{k}^{(i_r)} \underset{k \rightarrow \infty}{\longrightarrow} 0.
\end{equation}
As pointed out in \cite[Remark 1.33]{kenopesp1} in the case of pairings, it is always possible to decompose a given partition $\pi \in NC_{\geq 2+}\left(n_{i_1} \otimes \cdots \otimes n_{i_r}\right)$ into a disjoint union of connected partitions $\pi = \pi_1 \sqcup \cdots \sqcup \pi_m$, where for any $q \in \left\lbrace 1, \ldots , m \right\rbrace $, $\pi_q \in NC_{\geq 2}^{c}\left(\bigotimes_{j \in I_q}n_{i_j}\right)$ and where $I_1 \sqcup \cdots \sqcup I_m$ is a partition of the set $\left\lbrace 1, \ldots , r \right\rbrace $. Hence, the partition integral appearing in \eqref{hastogotozero} takes the form
\begin{equation}
\label{decompintoconnectedpartitions}
\int_{\pi}f_{k}^{(i_1)} \otimes \cdots \otimes f_{k}^{(i_r)} = \prod_{q=1}^{m}\int_{\pi_q}\bigotimes_{j \in I_q}f_{k}^{\left( i_j\right) }.
\end{equation}
As $\pi \in NC_{\geq 2+}\left(n_{i_1} \otimes \cdots \otimes n_{i_r}\right)$, $\pi$ contains at least one block $V_{*}$ of size greater or equal to three. Assume that this block $V_{*}$ belongs to the connected partition $\pi_{q^{*}}$ for a certain $q^{*} \in \left\lbrace 1, \ldots , m \right\rbrace $. This implies that $\pi_{q^{*}} \in NC_{\geq 2}^{c}\left(\bigotimes_{j \in I_{q^*}}n_{i_j}\right)$ where $\vert I_{q^*} \vert \geq 3 $. Hence, using \eqref{assumpintheproof} along with Proposition \ref{convzerosecondpart}, we get that 
\begin{equation*}
\int_{\pi_{q^*}}\bigotimes_{j \in I_{q^*}}f_{k}^{\left( i_j\right) } \underset{k \rightarrow \infty}{\longrightarrow} 0,
\end{equation*}
which, in view of \eqref{decompintoconnectedpartitions}, concludes the proof.

\section{Auxiliary results}
\label{auxsection}
\noindent A straightforward generalization of \cite[Theorem 3.15]{bope1} yields the following diagram formula for free Poisson multiple integrals.
\begin{proposition}
\label{diagramformula}
Let $d \geq 2$ and $n_1, \ldots , n_d$ be positive integers and suppose that $f_1,\ldots ,f_d$ are tamed functions with $f_i \in L^{2}\left(\R_{+}^{n_i}\right)$ for all $1 \leq i \leq d$. Then it holds that 
\begin{equation*}
	\varphi\left[I_{n_1}^{\hat{N}}\left(f_{1}\right) \cdots I_{n_d}^{\hat{N}}\left(f_{d}\right)\right] = \sum_{\pi \in NC_{\geq 2}\left(n_1 \otimes \cdots \otimes n_d\right)}\int_{\pi}f_1 \otimes \cdots \otimes f_d.
\end{equation*}
\end{proposition}

\begin{proposition}
\label{convzerosecondpart}
Let $d \geq 2$ and $n_1, \ldots , n_d$ be positive integers. For each $i=1,\dots,d$, let $\curly{f_{k}^{(i)} \colon k \geq 1}$ be a sequence of tamed mirror-symmetric functions in $L^{2}\left(\R_{+}^{n_i}\right)$ such that
\begin{equation}
\label{contgotozero}
	f_{k}^{(i)}\star_{1}^{0}f_{k}^{(i)} \underset{k \rightarrow \infty}{\longrightarrow} 0  \quad \mbox{in}\quad   L^{2}\left(\R_{+}^{2n_i - 1}\right)                        \qquad \mbox{and}         \qquad                             f_{k}^{(i)}\star_{\ell+p}^{\ell}f_{k}^{(i)} \underset{k \rightarrow \infty}{\longrightarrow} 0 \quad \mbox{in}\quad   L^{2}\left(\R_{+}^{2n_i - 2\ell - p}\right),
\end{equation}
for any $\ell \in \left\lbrace 1,\ldots ,n_i-1\right\rbrace $ and $p \in \left\lbrace 0,1\right\rbrace $. Then, for any $r \geq 3$ and any partition $\pi \in NC_{\geq 2}^{c}\left(n_{i_1}\otimes \cdots \otimes n_{i_r}\right)$, it holds that 
\begin{equation*}
	\int_{\pi}f_{k}^{(i_1)}\otimes \cdots \otimes f_{k}^{(i_r)} \underset{k \rightarrow \infty}{\longrightarrow} 0.
\end{equation*}
\end{proposition}
\begin{proof}
As in Definition \ref{defpartition1}, we denote by $B_1, \ldots, B_r$ the blocks of the partition $n_{i_1}\otimes \cdots \otimes n_{i_r}$. Furthermore, we denote by $b_{j}^{i}$ the elements of the block $B_i \in n_{i_1}\otimes \cdots \otimes n_{i_r}$, $1 \leq i \leq r$ and $1 \leq j \leq n_i$. Remark 3.10 in \cite{bope1}, along with the fact that $\pi$ respects $n_{i_1}\otimes \cdots \otimes n_{i_r}$, ensures that $\pi$ contains at least one block $V$ of the form $V = \left\lbrace b_{n_s}^{s},b_{1}^{s+1} \right\rbrace $ for some $s \in \left\lbrace 1, \ldots, r-1\right\rbrace $. Note that according to the terminology introduced in Definition \ref{deflink}, the block $V \in \pi$ links the blocks $B_s$ and $B_{s+1}$ of $n_{i_1}\otimes \cdots \otimes n_{i_r}$. Denote by $\ell$ the number of blocks of $\pi$ of size two (including $V$) linking $B_s$ and $B_{s+1}$. Similarly, denote by $p$ the number of blocks of size strictly greater than two linking $B_s$ and $B_{s+1}$. Note that, as $\pi$ is non-crossing, it is necessarily the case that $p \in \left\lbrace 0,1 \right\rbrace $.
\\~\\
If $\ell = 1$, there is only one block of size two (namely $V$) linking $B_s$ and $B_{s+1}$. If $\ell > 1$, then as $\pi$ is non-crossing and does not contain any singleton, the $\ell$ blocks of size exactly two are necessarily given by 
\begin{equation*}
V,\ \left\lbrace b_{n_s - 1}^{s}, b_{2}^{s+1}\right\rbrace, \ldots , \left\lbrace b_{n_s - \ell +1}^{s}, b_{\ell}^{s+1}\right\rbrace.
\end{equation*}
Observe that, in order for the above blocks to make sense, it necessarily holds that $1 \leq\ell \leq n_s \wedge n_{s+1}$ in the case where $p=0$ and $1 \leq\ell \leq n_s \wedge n_{s+1} -1$ in the case where $p=1$. Additionally, the fact that $\pi$ connects $n_{i_1}\otimes \cdots \otimes n_{i_r}$ excludes the case where $p=0$, $n_s = n_{s+1}$  and $\ell = n_{s}$. This implies that $\ell$ is necessarily such that $1 \leq \ell \leq n_s \wedge n_{s+1} - \delta_{n_s}^{n_{s+1}}$ if $p=0$ and $1 \leq \ell \leq n_s \wedge n_{s+1} - 1$ if $p=1$ (where $\delta_{n_s}^{n_{s+1}}$ denotes the Kronecker delta between $n_{s}$ and $n_{s+1}$).
\\~\\
Furthermore, observe that the number of blocks of $\pi$ linking $B_s$ to other blocks of $n_{i_1}\otimes \cdots \otimes n_{i_r}$ but not to $B_{s+1}$ is given by $n_s - \ell - p$. Similarly, the number of blocks of $\pi$ linking $B_{s+1}$ to other blocks of $n_{i_1}\otimes \cdots \otimes n_{i_r}$ but not to $B_{s}$ is given by $n_{s+1} - \ell - p$.
\\~\\
Having now a clear view of how the blocks $B_s$ and $B_{s+1}$ can be linked together and how they can be linked to the other blocks of $n_{i_1}\otimes \cdots \otimes n_{i_r}$ allows to specify further the form of the partition function $\left( f_{k}^{(i_1)}\otimes \cdots \otimes f_{k}^{(i_r)} \right)_{\pi}$ in $\vert \pi \vert$ variables obtained by identifying the variables $t_i$ and $t_j$ in the argument of the tensor
\begin{equation*}
f_{k}^{(i_1)}\otimes \cdots \otimes f_{k}^{(i_r)}\left( t_1 , \ldots , t_{n_{i_1}+ \cdots + n_{i_r}}\right) = \prod_{p=1}^{r}f_{k}^{(i_p)}\left(t_{n_{i_1} + \cdots + n_{i_{p-1}}+1}, \ldots, t_{n_{i_1} + \cdots + n_{i_{p}}} \right) 
\end{equation*}
if and only if $i$ and $j$ are in the same block of $\pi$. More specifically, we can write, for any $\ell, p$ as above, 
\begin{eqnarray*}
&& \left( f_{k}^{(i_1)}\otimes \cdots \otimes f_{k}^{(i_r)} \right)_{\pi}\left( \cdot, t_1, \ldots , t_{n_s - \ell - p}, \gamma_p, z_1, \ldots , z_{\ell}, s_1, \ldots , s_{n_{s+1} - \ell -p}\right) \\
&& \qquad\qquad\qquad = \mathscr{G}_{k}^{\pi , \ell , p}\left( \cdot, t_1, \ldots , t_{n_s - \ell - p}, \gamma_p, s_1, \ldots , s_{n_{s+1} - \ell -p}\right) \\
&& \qquad\qquad\qquad\qquad\qquad\qquad\times f_{k}^{(i_s)}\left(t_{n_s - \ell - p} , \ldots , t_1, \gamma_p , z_{\ell}, \ldots , z_1 \right) f_{k}^{(i_{s+1})}\left(z_1, \ldots , z_{\ell}, \gamma_p , s_1, \ldots , s_{n_{s+1} - \ell - p} \right),
\end{eqnarray*}
where $\mathscr{G}_{k}^{\pi , \ell , p}$ denotes a function of $\vert \pi \vert - \ell$ variables and where the dot in the arguments of $\left( f_{k}^{(i_1)}\otimes \cdots \otimes f_{k}^{(i_r)} \right)_{\pi}$ and $\mathscr{G}_{k}^{\pi , \ell , p}$ stands for the remaining $\vert \pi \vert - n_{s} - n_{s+1} + 2p + \ell -1$ variables. Observe that, using this decomposition along with the definition of star contractions given in Definition \ref{defcontractions}, it holds that 
\begin{equation*}
\int_{\R_{+}^{\ell}}\left( f_{k}^{(i_1)}\otimes \cdots \otimes f_{k}^{(i_r)} \right)_{\pi}d\lambda^{\ell} = \mathscr{G}_{k}^{\pi , \ell , p} \times f_{k}^{(i_s)}\star_{\ell+p}^{\ell}f_{k}^{(i_{s+1})},
\end{equation*}
so that, recalling the definition of partition integrals given in Definition \ref{defpartitionfunctions},
\begin{eqnarray*}
	&& \int_{\pi}f_{k}^{(i_1)}\otimes \cdots \otimes f_{k}^{(i_r)}  = \int_{\R_{+}^{\vert \pi \vert}}\left( f_{k}^{(i_1)}\otimes \cdots \otimes f_{k}^{(i_r)} \right)_{\pi} d\lambda ^{\vert \pi \vert} = \int_{\R_{+}^{\vert \pi \vert - \ell}}\mathscr{G}_{k}^{\pi , \ell , p} \times f_{k}^{(i_s)}\star_{\ell+p}^{\ell}f_{k}^{(i_{s+1})} d\lambda ^{\vert \pi \vert - \ell} \\
	&& = \int_{\R_{+}^{n_s + n_{s+1} - 2\ell - p}}f_{k}^{(i_s)}\star_{\ell+p}^{\ell}f_{k}^{(i_{s+1})} \int_{\R_{+}^{\vert \pi \vert + \ell + p - n_s - n_{s+1}}}\mathscr{G}_{k}^{\pi , \ell , p} d\lambda ^{\vert \pi \vert + \ell + p - n_s - n_{s+1}}  d\lambda ^{n_s + n_{s+1} - 2\ell - p} \\
	&& \leq \norm{f_{k}^{(i_s)}\star_{\ell+p}^{\ell}f_{k}^{(i_{s+1})}}_{L^2\left( \R_{+}^{n_s + n_{s+1} - 2\ell - p}\right) }\sqrt{\int_{\R_{+}^{n_s + n_{s+1} - 2\ell - p}} \abs{\int_{\R_{+}^{\vert \pi \vert + \ell + p - n_s - n_{s+1}}}\mathscr{G}_{k}^{\pi , \ell , p} d\lambda ^{\vert \pi \vert + \ell + p - n_s - n_{s+1}}}^2  d\lambda ^{n_s + n_{s+1} - 2\ell - p}}.
\end{eqnarray*}
Observe that the tameness of the sequences $\curly{f_{k}^{(i)} \colon k \geq 1}$, $1 \leq i \leq d$, ensures that the sequence 
\begin{equation*}
\curly{\sqrt{\int_{\R_{+}^{n_s + n_{s+1} - 2\ell - p}} \abs{\int_{\R_{+}^{\vert \pi \vert + \ell + p - n_s - n_{s+1}}}\mathscr{G}_{k}^{\pi , \ell , p} d\lambda ^{\vert \pi \vert + \ell + p - n_s - n_{s+1}}}^2  d\lambda ^{n_s + n_{s+1} - 2\ell - p}} \colon k \geq 1}
\end{equation*}
is bounded.
\\~\\
It remains to show that the contraction norm $\norm{f_{k}^{(i_s)}\star_{\ell+p}^{\ell}f_{k}^{(i_{s+1})}}_{L^2\left( \R_{+}^{n_s + n_{s+1} - 2\ell - p}\right) }$ converges to zero as $k$ converges to infinity. Using the mirror-symmetry of the functions $f_{k}^{(i_s)}$ and $f_{k}^{(i_{s+1})}$, it is easy to verify that 
\begin{equation*}
\norm{f_{k}^{(i_s)}\star_{\ell+p}^{\ell}f_{k}^{(i_{s+1})}}_{L^2\left( \R_{+}^{n_s + n_{s+1} - 2\ell - p}\right) }^{2} = \innerprod{f_{k}^{(i_s)}\star_{n_s-\ell}^{n_s-\ell - p}f_{k}^{(i_{s})}   ,   f_{k}^{(i_{s+1})}\star_{n_{s+1}-\ell}^{n_{s+1}-\ell - p}f_{k}^{(i_{s+1})}}_{L^2\left( \R_{+}^{2\ell + p}\right) },
\end{equation*}
so that
\begin{equation}
\label{upperboundsurlanormecont}
\norm{f_{k}^{(i_s)}\star_{\ell+p}^{\ell}f_{k}^{(i_{s+1})}}_{L^2\left( \R_{+}^{n_s + n_{s+1} - 2\ell - p}\right) } \leq \norm{f_{k}^{(i_s)}\star_{n_s-\ell}^{n_s-\ell - p}f_{k}^{(i_{s})} }_{L^2\left( \R_{+}^{2\ell + p}\right) }^{\frac{1}{2}}\norm{f_{k}^{(i_{s+1})}\star_{n_{s+1}-\ell}^{n_{s+1}-\ell - p}f_{k}^{(i_{s+1})} }_{L^2\left( \R_{+}^{2\ell + p}\right) }^{\frac{1}{2}}.
\end{equation}
The last step of this proof will be to show that the contractions appearing in \eqref{upperboundsurlanormecont} are always well defined and that the right-hand side of \eqref{upperboundsurlanormecont} always converges to zero as $k$ goes to infinity in view of the given assumptions. Begin by considering the case where $p=0$ and $n_s = n_{s+1}$. In this case, it holds that $1 \leq \ell \leq n_s -1$ and hence $1 \leq n_s - \ell \leq n_s - 1$. This implies that all the contractions appearing in \eqref{upperboundsurlanormecont} are well defined and that, in view of \eqref{contgotozero}, it holds that 
\begin{equation*}
\norm{f_{k}^{(i_s)}\cont{\ell} f_{k}^{(i_{s+1})}}_{L^2\left( \R_{+}^{2n_s  - 2\ell}\right) } \leq \norm{f_{k}^{(i_s)}\cont{n_s-\ell} f_{k}^{(i_{s})} }_{L^2\left( \R_{+}^{2\ell}\right) }^{\frac{1}{2}}\norm{f_{k}^{(i_{s+1})}\cont{n_{s}-\ell} f_{k}^{(i_{s+1})} }_{L^2\left( \R_{+}^{2\ell  }\right) }^{\frac{1}{2}} \underset{k \rightarrow \infty}{\longrightarrow} 0.
\end{equation*}
In the case where $p=0$, $n_s \neq n_{s+1}$ and $1 \leq \ell \leq n_{s} \wedge n_{s+1} -1$, it holds that $1 \leq n_{s} - \left(n_{s} \wedge n_{s+1}\right) + 1 \leq n_{s} - \ell \leq n_s - 1$ and  $1 \leq n_{s+1} - \left(n_{s} \wedge n_{s+1}\right) + 1 \leq n_{s+1} - \ell \leq n_{s+1} - 1$. This implies that all the contractions appearing in \eqref{upperboundsurlanormecont} are well defined and that, in view of \eqref{contgotozero}, it holds that 
\begin{equation*}
\norm{f_{k}^{(i_s)}\cont{\ell} f_{k}^{(i_{s+1})}}_{L^2\left( \R_{+}^{n_s + n_{s+1}  - 2\ell}\right) } \leq \norm{f_{k}^{(i_s)}\cont{n_s-\ell} f_{k}^{(i_{s})} }_{L^2\left( \R_{+}^{2\ell}\right) }^{\frac{1}{2}}\norm{f_{k}^{(i_{s+1})}\cont{n_{s+1} -\ell} f_{k}^{(i_{s+1})} }_{L^2\left( \R_{+}^{2\ell  }\right) }^{\frac{1}{2}} \underset{k \rightarrow \infty}{\longrightarrow} 0.
\end{equation*}
In the case where $p=0$, $n_s \neq n_{s+1}$ and $\ell = n_{s} \wedge n_{s+1}$, assume without loss of generality that $n_s < n_{s+1}$. This yields $n_s - \ell = 0$ and $1 \leq n_{s+1} - \ell = n_{s+1} - n_{s} \leq  n_{s+1} - 1$. This implies that all the contractions appearing in \eqref{upperboundsurlanormecont} are well defined and that, in view of \eqref{contgotozero} and the tameness of the sequence $\left\lbrace f_{k}^{(i_s)} \colon k\geq 1 \right\rbrace $, it holds that 
\begin{eqnarray*}
\norm{f_{k}^{(i_s)}\cont{n_{s}} f_{k}^{(i_{s+1})}}_{L^2\left( \R_{+}^{n_s + n_{s+1}  - 2\ell}\right) } &\leq & \norm{f_{k}^{(i_s)} \otimes f_{k}^{(i_{s})} }_{L^2\left( \R_{+}^{2n_{s}}\right) }^{\frac{1}{2}}\norm{f_{k}^{(i_{s+1})}\cont{n_{s+1} -n_{s}} f_{k}^{(i_{s+1})} }_{L^2\left( \R_{+}^{2n_{s} }\right) }^{\frac{1}{2}} \\
& = &   \norm{f_{k}^{(i_s)}}_{L^2\left( \R_{+}^{n_{s}}\right) }\norm{f_{k}^{(i_{s+1})}\cont{n_{s+1} -n_{s}} f_{k}^{(i_{s+1})} }_{L^2\left( \R_{+}^{2n_{s} }\right) }^{\frac{1}{2}}  \underset{k \rightarrow \infty}{\longrightarrow} 0.
\end{eqnarray*}
Finally, in the case where $p=1$, it is always the case that $1 \leq \ell \leq n_{s} \wedge n_{s+1} -1$ so that $0 \leq n_{s} - \left(n_{s} \wedge n_{s+1} \right) \leq n_{s } - \ell - 1 \leq n_{s} - 2$ and $0 \leq n_{s+1} - \left(n_{s} \wedge n_{s+1} \right) \leq n_{s +1} - \ell - 1 \leq n_{s+1} - 2$. This implies that all the contractions appearing in \eqref{upperboundsurlanormecont} are well defined and that, in view of \eqref{contgotozero}, it holds that 
\begin{equation*}
\norm{f_{k}^{(i_s)}\star_{\ell+1}^{\ell}f_{k}^{(i_{s+1})}}_{L^2\left( \R_{+}^{n_s + n_{s+1} - 2\ell - 1}\right) } \leq \norm{f_{k}^{(i_s)}\star_{n_s-\ell}^{n_s-\ell - 1}f_{k}^{(i_{s})} }_{L^2\left( \R_{+}^{2\ell + 1}\right) }^{\frac{1}{2}}\norm{f_{k}^{(i_{s+1})}\star_{n_{s+1}-\ell}^{n_{s+1}-\ell - 1}f_{k}^{(i_{s+1})} }_{L^2\left( \R_{+}^{2\ell + 1}\right) }^{\frac{1}{2}} \underset{k \rightarrow \infty}{\longrightarrow} 0.
\end{equation*}
\end{proof}

\end{document}